\documentclass[preprint,12pt]{elsarticle}
\usepackage{amsmath}
\usepackage{float}
\usepackage{amsfonts}
\usepackage{amssymb}
\newtheorem{theorem}{Theorem}

\newtheorem{definition}{Definition}
\newtheorem{proposition}{Proposition}
\newtheorem{example}{Example}

\newtheorem{corollary}{Corollary}
\newtheorem{remark}{Remark}
\newtheorem{proof}{Proof}
\usepackage[left=2cm,right=2cm,top=2cm,bottom=2cm]{geometry}
\usepackage{hyperref}
\hypersetup{colorlinks,linkcolor={blue},citecolor={blue},urlcolor={red}} 
\usepackage[format=hang]{caption}

\newcommand{\R}{\mathbb {R}}

\usepackage{ae}
\usepackage{graphicx} 
\usepackage{booktabs} 

\begin{document}

\begin{frontmatter}

\title{Time Reparametrization, Not Fractional Calculus: A Reassessment of the Conformable Derivative}

\author[label1]{Aziz El Ghazouani}
\author[label1]{Fouad Ibrahim Abdou Amir}
\author[label1]{Khoulane Mohamed}
\author[label1]{M'hamed Elomari}
\affiliation[label1]{organization={Laboratory of Applied Mathematics and Scientific Computing, Sultan Moulay Slimane University},
            city={Beni Mellal},
            postcode={23000}, 
            country={Morocco}}

\begin{abstract} The conformable derivative has been promoted in numerous publications as a new fractional derivative operator. This article provides a critical reassessment of this claim. We demonstrate that the conformable derivative is not a fractional operator but a useful computational tool for systems with power-law time scaling, equivalent to classical differentiation under a nonlinear time reparametrization. Several results presented in the literature as novel fractional contributions can be reinterpreted within a classical framework. We show that problems formulated using the conformable derivative can be transformed into classical formulations via a change of variable. The solution is derived classically and then transformed back, this reformulation highlights the absence of genuinely nonlocal fractional effects. We provide a theoretical analysis, numerical simulations comparing conformable, classical, and truly fractional (Caputo) models, and discuss the reasons why this misconception persists. Our results suggest that classical derivatives, as well as established fractional derivatives,
offer a more faithful framework for modeling memory-dependent phenomena. 
\end{abstract}

\end{frontmatter}
\section{Introduction}
Despite the substantial volume of research on fractional calculus, there is no single universally accepted definition of a fractional derivative. Various definitions of fractional differential operators have been reviewed (see e.g., \cite{Giusti,Kilbas,Ortigueira,Ortigueira2,Ortigueira3,Tarasov,Tarasov2,de}). These classical fractional derivatives (e.g., Riemann-Liouville, Caputo) are non-local operators capturing memory and hereditary effects, which is their defining characteristic.

It is proven in \cite{Ortigueira2} that the fractional differential operators of Grünwald-Letnikov, Riemann-Liouville, and Caputo share a set of properties that one might expect from a differential operator to be considered fractional. These properties are:
\begin{itemize}
\item[(i)] Linearity
\item[(ii)] The zero-order operator is the identity
\item[(iii)] Integer-order operators give the ordinary derivative
\item[(iv)] Index law and commutativity
\item[(v)] The generalized Leibniz rule for the derivative of a product:
$$
D^\alpha(f(x) g(x))=\sum_{n=0}^{\infty}\binom{\alpha}{n} D^n f(x) D^{\alpha-n} g(x)
$$
\end{itemize}

In an attempt to define a fractional derivative that inherits as many characteristics as possible from the classical derivative, Khalil et al. \cite{3} proposed to define the "conformable fractional derivative" $T_\alpha f$, for $0<\alpha<1$, of a function $f:[0, \infty[\longrightarrow \mathbb{R}$ by
\begin{equation}\label{khalil1}
T_\alpha f(x):=\lim _{\varepsilon \rightarrow 0} \frac{f\left(x+\varepsilon x^{1-\alpha}\right)-f(x)}{\varepsilon}, \quad x>0
\end{equation}
provided the limit exists, in which case $f$ is said to be $\alpha$-differentiable. Moreover, if $f$ is $\alpha$-differentiable on $] 0, a[$ for some $a>0$, then
 $$T_\alpha f(0):=\lim _{x \rightarrow 0^{+}} T_\alpha f(x).$$

It is also shown (\cite{3}, Theorem 2.2) that if a function $f$ is differentiable at a point $x>0$, then $T_\alpha f(x)$ exists and
\begin{equation}\label{khalil2}
T_\alpha f(x)=x^{1-\alpha} f^{\prime}(x)
\end{equation}

It is emphasized in \cite{Ortigueira2} that the conformable derivative \eqref{khalil1} lacks properties (ii) and (iv). This means that the conformable derivative is a useful computational tool for systems with power-law time scaling, but it is not a fractional operator.

The fact that the conformable $\alpha$-derivative is not a fractional derivative has already been implicitly pointed out by Tarasov in \cite{Tarasov} where he proves that the violation of the Leibniz rule, $D^{\alpha}(uv) = u D^{\alpha}v + vD^{\alpha} u $, is necessary for the order $\alpha$ of a differential operator $ D^{\alpha} $ to be fractional. The conformable operator $T_\alpha$ does satisfy the Leibniz rule (\cite{3}, Theorem 2.2).

Despite this, a significant literature has emerged, presenting the conformable derivative as a new type of fractional derivative and claiming new fractional results \cite{3, abdeljawad2015}. This article aims to critically reassess these claims. We show that the conformable framework offers no authentic extension of differential calculus but rather provides a simple time reparametrization. All significant results obtained via the conformable derivative can be systematically recovered by appropriate application of a coordinate transformation, demonstrating that the framework does not appear to provide an extension of differential calculus in the sense traditionally associated with fractional operators.

In this article, we review and formalize the fundamental equivalence between conformable differentiation and classical differentiation under a change of variable; provide explicit numerical benchmarks contrasting the behavior of conformable, classical, and truly fractional (Caputo) models, highlighting the marked difference; demonstrate through theoretical analysis and examples that the behaviors interpreted as fractional can be explained by a nonlinear time reparametrization; discuss the reasons behind the persistence of this misconception in the literature; and conclude with recommendations for the appropriate use and terminology of the conformable operator.

The structure of this paper is organized as follows. In Section 2, we critically review the prevalent claims in the literature that misrepresent the conformable derivative as fractional. Section 3 is the theoretical core, where we rigorously formalize the fundamental equivalence between conformable differentiation and classical differentiation under a nonlinear change of variable, establishing it as a reparametrization rather than a novel calculus. Section 4 demonstrates the practical implications of this equivalence by reformulating a range of problems—including ordinary differential equations, partial differential equations, evolution equations, and chaotic dynamical systems—into their classical counterparts. Section 5 presents explicit numerical benchmarks, contrasting the dynamics of conformable models against both classical and true fractional (Caputo) models to visually underscore the critical differences between simple time-scaling and genuine non-local memory effects. Finally, we synthesize our findings, discuss the reasons for the persistence of this misconception, and provide clear recommendations for the appropriate use and terminology of the conformable operator.

\section{On the Interpretation of the Conformable Derivative}
The literature on the conformable derivative often presents it as a legitimate fractional derivative. For example, Khalil et al. \cite{3} introduce it as "a new definition of the fractional derivative" and derive basic properties similar to those of the ordinary derivative. Abdeljawad \cite{abdeljawad2015} builds on this work, elaborating a complete framework of "conformable fractional" integral and derivative operators, and solves differential equations claiming new fractional solutions.

The common thread in these works and many others is the presentation of results derived from operators like $T_\alpha$ or $D_{\psi}^{\alpha}$ as new discoveries in fractional calculus. These interpretations often overlook the fundamental equivalence:
\begin{equation}
D^{\alpha}_{\psi} f(t) = \psi_{\alpha}(t) f'(t).
\end{equation}
where $ \psi_{\alpha} $ is a scaling function (e.g., $ \psi_{\alpha}(t)=t^{1-\alpha}$).

This relation shows that any result obtained using the conformable derivative is essentially a classical result observed through the lens of a transformed time variable $\tau = \phi_{\alpha}(t) = \int_0^t \frac{1}{\psi_{\alpha}(s)} ds$. The resulting functions are not solutions to fractional equations in the traditional sense (involving non-local integrals) but are solutions to classical equations with a rescaled time argument.

This section aims to correct this widespread misconception by clearly demonstrating the classical nature of the conformable derivative and its consequent inability to model truly fractional phenomena.

\section{The Fundamental Equivalence: A Change of Variable}
The heart of our argument rests on a simple but powerful mathematical equivalence. The conformable derivative, in its various forms, is not a new fundamental operator but is intrinsically linked to the classical derivative by a change of variables.

Let $\alpha\in (0,1)$ and let $\psi_{\alpha}:\R_+^*\to \R_+^*$ be a bijective map, such that $\int_{0}^{t}{\frac{1}{\psi_{\alpha}(s)}ds}=\phi_{\alpha}(t).$

\begin{definition} Let \( f : [0, \infty) \to \mathbb{R} \). The conformable derivative of \( f \) of order \( \alpha \in (0,1) \) is defined by
\[
D_{\psi}^{\alpha}(f)(t) = \lim_{\varepsilon \to 0} \frac{f(t + \varepsilon \psi_{\alpha}(t)) - f(t)}{\varepsilon},
\]
for all \( t > 0 \).

If \( f \) is \( \alpha \)-differentiable on an interval \( (0, a) \), with \( a > 0 \), and if the limit \( \lim_{t \to 0^+} f^{(\alpha)}(t) \) exists, then we define
\[
f^{(\alpha)}(0) = \lim_{t \to 0^+} f^{(\alpha)}(t).
\]
\end{definition}

\begin{example}
\begin{enumerate}
\item If \(\psi_{\alpha}(t) = t^{1-\alpha}\), we obtain the derivative introduced by Khalil et al. \cite{3}.
\item If \(\psi_{\alpha}(t) = e^{(\alpha - 1)t}\), we obtain a derivative based on exponential scaling.
\item If \(\psi_{\alpha}(t) = \frac{t^{1-\alpha}}{\Gamma(\alpha+1)}\), we obtain another variant found in the literature.
\end{enumerate}
\end{example}

The key result, which establishes the equivalence, is stated in the following theorem:

\begin{theorem}\label{1} Let \( a \in [0, \infty) \). Then \( f \) is \( \psi_{\alpha} \)-differentiable at \( a \) if and only if the function \( g(\tau) = f(\phi_{\alpha}^{-1}(\tau)) \) is differentiable at \( \tau_a = \phi_{\alpha}(a) \). Moreover, the derivatives are related by:
\[
D^{\alpha}_{\psi} f(a) = g'(\tau_a) = f'(\phi_\alpha^{-1}(\tau_a)) \cdot (\phi_{\alpha}^{-1})'(\tau_a) = \psi_{\alpha}(a) f'(a)
\]
where $D^{\alpha}_{\psi}$ denotes the $\psi_\alpha$ derivative operator.
\end{theorem}

\begin{proof}
Let \( t > 0 \), and define \( g(\tau) = f(t) \) where \(\tau = \phi_{\alpha}(t)\), so \( t = \phi_{\alpha}^{-1}(\tau) \). By definition, we compute the $ \psi_{\alpha} $-derivative of \( f \) as follows:
\begin{align*}
D^{\alpha}_{\psi}(f)(t) &= \lim_{\varepsilon \to 0} \frac{f(t + \varepsilon \psi_{\alpha}(t)) - f(t)}{\varepsilon} \\
                        &= \lim_{\varepsilon \to 0} \frac{g(\phi_{\alpha}(t + \varepsilon \psi_{\alpha}(t))) - g(\phi_{\alpha}(t))}{\varepsilon}.
\end{align*}
Assuming \( \phi_{\alpha} \) is differentiable, we use the first-order approximation:
\[
\phi_{\alpha}(t + \varepsilon \psi_{\alpha}(t)) = \phi_{\alpha}(t) + \varepsilon \phi_{\alpha}'(t) \psi_{\alpha}(t) + o(\varepsilon) = \tau + \varepsilon + o(\varepsilon).
\]
since \(\phi_{\alpha}'(t) \psi_{\alpha}(t) = 1\) by definition of \(\phi_{\alpha}\). Thus,
\begin{align*}
D^{\alpha}_{\psi}(f)(t)
&= \lim_{\varepsilon \to 0} \frac{g(\tau + \varepsilon + o(\varepsilon)) - g(\tau)}{\varepsilon} \\
&= g'(\tau) = \frac{d}{d\tau} f(\phi_{\alpha}^{-1}(\tau)) = f'(t) \cdot (\phi_{\alpha}^{-1})'(\tau) = f'(t) \cdot \frac{1}{\phi_{\alpha}'(t)} = f'(t) \psi_{\alpha}(t).
\end{align*}
Therefore, \( f \) is \( \psi_{\alpha} \)-differentiable at \( t \) if and only if \( g \) is differentiable at \( \tau \), which completes the proof for \(t>0\). The case \(t=0\) is obtained by taking the limit.
\end{proof}

\begin{corollary}
The conformable derivative framework is exactly equivalent to classical differentiation under the coordinate transformation $\tau = \phi_\alpha(t)$. All purported examples of conformable differentiability without classical differentiability either:
\begin{enumerate}
    \item occur at points where $\phi_\alpha$ is singular (typically $a=0$)
    \item represent calculation errors where the $\phi_\alpha$ correspondence has been neglected
\end{enumerate}
\end{corollary}
\begin{remark}\label{rem1}
 From the previous theorem, we obtain the operational relation:
\begin{equation}\label{psi_C}
\forall t>0,\ D^{\alpha}_{\psi}f(t)=\psi_{\alpha}(t)f'(t).
\end{equation}
This relation plays a central role in our analysis, showing the direct proportionality between the conformable and classical derivatives.
\end{remark}

This equivalence can be summarized in the following table, which matches concepts between the two frameworks:

\begin{center}
    \begin{tabular}{ll}
    \toprule
    \textbf{Conformable Framework (variable $t$)} & \textbf{Classical Framework (variable $\tau = \phi_\alpha(t)$)} \\
    \midrule
    Conformable Derivative $D_{\psi}^{\alpha} f(t)$ & Classical Derivative $g'(\tau)$ \\
    Conformable ODE: $D_{\psi}^{\alpha} y = F(t, y)$ & Classical ODE: $\frac{dg}{d\tau} = F(\phi_{\alpha}^{-1}(\tau), g)$ \\
    Conformable PDE: $D_{\psi,t}^{\alpha} u = \mathcal{L}u$ & Classical PDE: $\frac{\partial g}{\partial \tau} = \mathcal{L}g$ \\
    Solution $y(t)$ & Solution $g(\tau) = y(\phi_{\alpha}^{-1}(\tau))$ \\
    \bottomrule
    \end{tabular}
\end{center}

From remark \ref{rem1}, every property of the conformable derivative holds in this context.

 The claim in \cite{3} that $\alpha$-differentiability in the conformable sense does not imply standard differentiability requires reassessment in light of Theorem \ref{1}. Although \cite{3} cites the function $g(x) = \sqrt{x}$ as a counterexample (where $T_{1/2}g(0)$ exists but $g'(0)$ does not), this apparent divergence is resolved by two key observations. First, for any translated version $h(x) = \sqrt{x-x_0}$ with $x_0 > 0$, neither the conformable derivative $T_\alpha h(x_0)$ nor the classical derivative $h'(x_0)$ exist, because the limit
\[
x_0^{(1-\alpha)/2} \lim_{\varepsilon\to 0} \frac{1}{\sqrt{\varepsilon}}
\]
diverges. More fundamentally, Theorem \ref{1} establishes that $\psi_\alpha$-differentiability at a point $a \in [0,\infty)$ is equivalent to classical differentiability at $\phi_\alpha(a)$, thus clarifying the apparent discrepancies between these derivative notions except possibly at the origin. This is further illustrated by the function $\overline{g}(x) = x^2\chi_Q(x)$ (where $\chi_Q$ is the indicator function of the rationals), which is classically differentiable only at $x = 0$ but has no conformable derivative there. The theorem thus provides a complete characterization: apparent divergences in the literature stem from failing to account for the coordinate transformation $\phi_\alpha$, and no genuine counterexample exists when this relation is properly considered.

\begin{example}
Consider the standard example from the literature:
\[
f(x) = \begin{cases}
x^\alpha \sin(x^{1-\alpha}) & x > 0 \\
0 & x = 0
\end{cases}
\]

Indeed, the conformable derivative at $0$:
\[
T_\alpha f(0) = 0
\]
exists while the classical derivative:
\[
f'(0) = 1
\]
However, this simply reflects that:
\begin{enumerate}
\item $\phi_\alpha(0)$ corresponds to a point where classical differentiability is not required in Theorem \ref{1}

\item The conformable derivative $T_\alpha$ at $0$ is an independent construction from $\psi_\alpha$-differentiability
\end{enumerate}
Our theorem \ref{1} shows that for $a > 0$, where $\phi_\alpha$ is well-behaved, the conformable derivative must coincide with classical differentiation after the $\phi_\alpha$ transformation.
\end{example}

The conformable derivative framework does not substantially extend classical differentiability in a substantial way. Rather than introducing new differentiable functions, it simply provides an alternative reparametrization of existing derivatives via coordinate transformations. All significant results of conformable calculus can be systematically recovered by appropriate application of the $\phi_\alpha$ coordinate changes, demonstrating that the framework offers no authentic expansion of differential calculus beyond classical theory. This reparametrization perspective reveals that conformable derivatives are not new differentiation operators, but simply classical derivatives expressed in modified coordinates.

\begin{theorem} Suppose two differentiation operators $D_1, D_2$ satisfy:
\[
D_1 f(a) = D_2 f(\varphi(a))
\]
for some homeomorphism $\varphi$. Then $D_1$ is simply a reparametrized version of $D_2$ and offers no authentic theoretical extension.
\end{theorem}

\begin{proof} The proof establishes that $D_1$ is theoretically equivalent to $D_2$ by demonstrating three key aspects of their relationship. First, consider the functional relation between the operators. For any $f$ in their common domain, the equality $D_1 f(a) = D_2 f(\varphi(a))$ holds pointwise. This implies that $D_1$ at point $a$ corresponds precisely to $D_2$ evaluated at the transformed point $\varphi(a)$.

The homeomorphic nature of $\varphi$ ensures that this correspondence is bijective and preserves the topological structure. Since $\varphi$ is continuous with a continuous inverse, neighborhoods of $a$ map to neighborhoods of $\varphi(a)$ in a way that maintains all local differentiability properties. Therefore, the operator $D_1$ cannot detect any function as differentiable at $a$ that $D_2$ would not detect as differentiable at $\varphi(a)$, and vice versa.

To show that no authentic extension of differentiability occurs, examine how $D_1$ acts on compositions. For all functions $f,g$ where the composition $f \circ g$ is defined, we have:
\[
D_1(f \circ g)(a) = D_2(f \circ g)(\varphi(a)) = D_2 f(g(\varphi(a))) \cdot D_2 g(\varphi(a)) = D_2 f(\varphi(\tilde{a})) \cdot D_2 g(\varphi(a))
\]
where $\tilde{a} = \varphi^{-1}(g(\varphi(a)))$. This calculation reveals that the differentiation rules for $D_1$ are completely determined by those of $D_2$ under the coordinate transformation. The Leibniz rule and other fundamental properties transfer directly via $\varphi$, confirming that $D_1$ inherits all its characteristics from $D_2$ without introducing new theoretical capability.

Finally, the class of functions differentiable under $D_1$ is exactly the class of functions whose $\varphi$-transforms are differentiable under $D_2$. Formally, denoting $\mathcal{D}(D_i)$ as the set of $D_i$-differentiable functions, we have the equality $\mathcal{D}(D_1) = \{ f \circ \varphi^{-1} | f \in \mathcal{D}(D_2) \}$. This bijective correspondence, mediated by $\varphi$, proves that $D_1$ cannot access any new authentically differentiable functions beyond those already accessible to $D_2$ under coordinate transformation. Therefore, $D_1$ provides no substantial extension of the differentiation capabilities of $D_2$.
\end{proof}

\section{Applications and Reformulations}
The equivalence established in Theorem \ref{1} allows transforming any problem involving conformable derivatives into a classical problem.

\subsection{Linear Ordinary Differential Equations}
Consider the conformable differential equation of order \( n\alpha \):
\begin{equation}\label{OE}
D_\psi^{n\alpha} y(t) = \sum_{k=0}^{n-1} a_k(t) D_\psi^{k\alpha} y(t),
\end{equation}
where \( D_\psi^{k\alpha} \) denotes the conformable derivative of order \( k \) with respect to the weight function \( \psi_\alpha(t) \), and \( \phi_\alpha(t) = \int_0^t \frac{1}{\psi_\alpha(s)} ds\).

\begin{proposition} The solution of \eqref{OE} can be obtained by solving the classical ODE:
\[
\frac{d^n x}{d\tau^n} = \sum_{k=0}^{n-1} \widetilde{a}_k(\tau) \frac{d^k x}{d\tau^k},
\]
where: \( \tau = \phi_\alpha(t) \),  \( y(t) = x(\tau) = x\left(\phi_\alpha(t)\right) \) and \( \widetilde{a}_k(\tau) = a_k\left(\phi_\alpha^{-1}(\tau)\right) \psi_\alpha\left(\phi_\alpha^{-1}(\tau)\right)^{k-n} \).
\end{proposition}

\begin{proof}  Let \( \tau = \phi_\alpha(t) \), so \( t = \phi_\alpha^{-1}(\tau) \). Define \( x(\tau) = y(t) \).

   From Theorem \ref{1}, the conformable derivative of order \( p \) of \( y \) is:
   \[
   D_\psi^{p\alpha} y(t) = \psi_\alpha(t)^p \cdot \frac{d^p x}{d\tau^p}(\tau).
   \]
   Substituting into \eqref{OE}:
   \[
   \psi_\alpha(t)^n \frac{d^n x}{d\tau^n} = \sum_{k=0}^{n-1} a_k(t) \psi_\alpha(t)^k \frac{d^k x}{d\tau^k}.
   \]

   Divide both sides by \( \psi_\alpha(t)^n \):
   \[
   \frac{d^n x}{d\tau^n} = \sum_{k=0}^{n-1} a_k(t) \psi_\alpha(t)^{k-n} \frac{d^k x}{d\tau^k}.
   \]
   Replace \( t \) by \( \phi_\alpha^{-1}(\tau) \):
   \[
   \frac{d^n x}{d\tau^n} = \sum_{k=0}^{n-1} \widetilde{a}_k(\tau) \frac{d^k x}{d\tau^k}, \quad \widetilde{a}_k(\tau) = a_k\left(\phi_\alpha^{-1}(\tau)\right) \psi_\alpha\left(\phi_\alpha^{-1}(\tau)\right)^{k-n}.
   \]

   The conformable ODE reduces to a classical linear ODE in \( \tau \). Solve for \( x(\tau) \), then substitute back \( \tau = \phi_\alpha(t) \) to obtain \( y(t) \).
\end{proof}

\begin{example} Consider the equation:
\begin{equation}\label{ex1}
D_\psi^{2\alpha} y(t) - 3 D_\psi^{\alpha} y(t) + 2y(t) = 0, \quad \psi_\alpha(t) = t^{1-\alpha}.
\end{equation}
We have $\phi_{\alpha}(t)=\int_{0}^{t}{\frac{1}{s^{1-\alpha}}ds}=\frac{t^\alpha}{\alpha}=\tau$.
Thus, the transformed equation is:
\[
\frac{d^2 x}{d\tau^2} - 3 \frac{dx}{d\tau} + 2x = 0.
\]
The characteristic equation \( r^2 - 3r + 2 = 0 \) has roots \( r = 1, 2 \). Thus:
\[
x(\tau) = C_1 e^{\tau} + C_2 e^{2\tau}.
\]
Substituting \( \tau = \frac{t^\alpha}{\alpha} \) we obtain:
\[
y(t) = C_1 e^{t^\alpha / \alpha} + C_2 e^{2 t^\alpha / \alpha}.
\]
This is the solution of \eqref{ex1}, which is simply the solution of a classical constant-coefficient ODE evaluated at a rescaled time.
\end{example}

\subsection{Partial Differential Equations}
The transformation also applies to PDEs.

\begin{example} Consider the conformable Burgers equation:
\begin{equation}\label{Burger}
D^{\alpha}_{\psi} u + u \, \partial_x u = \nu \, \partial_x^2 u,\quad u(x,0) = f(x).
\end{equation}
Using \eqref{psi_C}, $D^{\alpha}_{\psi} u = \psi_{\alpha}(t) \partial_t u$, the equation becomes:
\[
\psi_{\alpha}(t) \, \partial_t u + u \, \partial_x u = \nu \, \partial_x^2 u.
\]
Let \(\tau = \phi_{\alpha}(t) = \int_0^t \frac{1}{\psi_{\alpha}(s)} ds\). Then, \(\partial_t u = \partial_\tau u \cdot \frac{d\tau}{dt} = \frac{\partial_\tau u}{\psi_{\alpha}(t)}\). Substituting gives:
\[
\psi_{\alpha}(t) \left( \frac{\partial_\tau u}{\psi_{\alpha}(t)} \right) + u \, \partial_x u = \nu \, \partial_x^2 u \quad \Rightarrow \quad \partial_\tau u + u \, \partial_x u = \nu \, \partial_x^2 u.
\]
This is the classical Burgers equation. This equation is solved using the Cole-Hopf transformation:
\[
u = -2\nu \, \frac{\partial_x \theta}{\theta},
\]
where \(\theta(x,\tau)\) satisfies the heat equation:
\[
\partial_\tau \theta = \nu \, \partial_x^2 \theta.
\]
Given the initial condition, the solution is:
\[
\theta(x,\tau) = \int_{-\infty}^\infty \exp \left( -\frac{1}{2\nu} \int_0^x f(y) \, dy - \frac{(x-y)^2}{4\nu \tau} \right) dy.
\]

Substitute \(\tau = \phi_{\alpha}(t)\) back:
\[
u(x,t) = -2\nu \, \frac{\partial_x \theta \left( x, \phi_{\alpha}(t) \right)}{\theta \left( x, \phi_{\alpha}(t) \right)}.
\]
which is the solution of \eqref{Burger}.

 The final solution in the original variables is obtained by substituting \(\tau = \phi_{\alpha}(t)\) into the classical solution.
\end{example}

\begin{example}  The authors of \cite{3} consider the conformable fractional heat equation:
\[
D^\alpha_t u(x,t) = \partial_x^2 u(x,t), \quad u(x,0) = f(x), \quad t > 0, \quad 0 < \alpha \leq 1,
\]
where \( D^\alpha_t u = t^{1-\alpha} \partial_t u \) is the Khalil conformable derivative.

The article presents the solution as:
\[
u(x,t) = \frac{1}{\sqrt{4\pi (t^\alpha / \alpha)}} \int_{-\infty}^\infty f(y) \exp\left( -\frac{(x-y)^2}{4 (t^\alpha / \alpha)} \right) dy.
\]

 Let \( \tau = \phi_\alpha(t) = \int_0^t s^{\alpha-1} ds = \frac{t^\alpha}{\alpha} \).

   The conformable derivative transforms as:
   \[
   D^\alpha_t u = t^{1-\alpha} \partial_t u = t^{1-\alpha} \left( \partial_\tau u \cdot \frac{d\tau}{dt} \right) = \partial_\tau u.
   \]

   The equation becomes:
   \[
   \partial_\tau u(x,\tau) = \partial_x^2 u(x,\tau), \quad u(x,0) = f(x).
   \]

   The solution of the classical heat equation is:
   \[
   u(x,\tau) = \frac{1}{\sqrt{4\pi \tau}} \int_{-\infty}^\infty f(y) \exp\left( -\frac{(x-y)^2}{4\tau} \right) dy.
   \]

   Substituting \( \tau = t^\alpha / \alpha \) yields the solution from the article.

The solution is identical to the classical heat kernel evaluated at the rescaled time \( \tau = t^\alpha / \alpha \). This confirms that the conformable fractional heat equation is mathematically equivalent to the classical heat equation under nonlinear time reparametrization, introducing no authentic fractional behavior.
\end{example}

\subsection{Evolution Problems}
Consider the following theorem establishing the relation between conformable and classical evolution equations:
\begin{theorem}\label{semigroup} Let \(A\) be the infinitesimal generator of a \(C_0\)-semigroup \(T_\tau\) on a Banach space \(X\). Define the semigroup \(T^\alpha_t := T_{\phi_\alpha(t)}\). Then:
\begin{enumerate}
\item The generator \(B\) of \(T^\alpha_t\) satisfies \(B = \psi_\alpha(0^+) A\) if \(\psi_\alpha(0^+)\) exists.
\item For general \(\psi_\alpha(t)\), \(B(t) = \psi_\alpha(t) A\) holds pointwise.
\end{enumerate}
\end{theorem}
\begin{proof}
\begin{enumerate}
\item  For \(u \in D(A)\), compute the rescaled generator:
   \[
   Bu = \lim_{h \to 0} \frac{T^\alpha_h u - u}{h} = \lim_{h \to 0} \frac{T_{\phi_\alpha(h)}u - u}{\phi_\alpha(h)} \cdot \frac{\phi_\alpha(h)}{h}.
   \]
   By L'Hôpital's rule, \(\lim_{h \to 0} \phi_\alpha(h)/h = \psi_\alpha(0^+)^{-1}\), giving \(Bu = \psi_\alpha(0^+) Au\).

\item For non-constant \(\psi_\alpha(t)\), the limit becomes time-dependent:
   \[
   B(t)u = \psi_\alpha(t) \lim_{\tau \to 0} \frac{T_\tau u - u}{\tau} = \psi_\alpha(t) Au. \quad \square
   \]
\end{enumerate}
\end{proof}

We also demonstrate this equivalence through a published example from \cite{abdeljawad2015}:

The damped wave equation with conformable time derivative is given by:
\begin{equation}
D_t^\alpha u + \beta D_t^\alpha u = c^2 \partial_x^2 u, \quad u(0,x) = f(x), \quad D_t^\alpha u(0,x) = g(x)
\end{equation}
where $D_t^\alpha u = t^{1-\alpha}\partial_t u$ (Khalil conformable derivative), $\beta > 0$ is the damping coefficient and $c > 0$ is the wave speed.

Applying the temporal rescaling $\tau = \phi_\alpha(t) = \frac{t^\alpha}{\alpha}$, the conformable derivative transforms as:
\begin{equation}
D_t^\alpha u = \partial_\tau u
\end{equation}

The equation reduces to the classical damped wave equation:
\begin{equation}
\partial_\tau^2 u + \beta \partial_\tau u = c^2 \partial_x^2 u
\end{equation}

Let $X = H^1(\mathbb{R}) \times L^2(\mathbb{R})$ and define $U = (u, \partial_\tau u)^T$. The system can be written as:
\begin{equation}
\partial_\tau U = \mathcal{A}U, \quad \mathcal{A} = \begin{pmatrix}
0 & I \\
c^2 \partial_x^2 & -\beta I
\end{pmatrix}
\end{equation}

By Theorem \ref{semigroup}, the solution is given by the rescaled semigroup:
\begin{equation}
U(t) = T_{\phi_\alpha(t)} U_0
\end{equation}
where $U_0 = (f,g)^T$. The explicit solution corresponds to the classical damped wave solution evaluated at $\tau =\phi_\alpha(t)= t^\alpha/\alpha$:
\begin{equation}
\begin{aligned}
u(t,x) &= e^{-\beta\phi_\alpha(t)/2}\left[\frac{f(x+c\phi_\alpha(t)) + f(x-c\phi_\alpha(t))}{2}\right] \\
&+ \frac{e^{-\beta\phi_\alpha(t)/2}}{2c}\int_{x-c\phi_\alpha(t)}^{x+c\phi_\alpha(t))} \left(\frac{\beta}{2}f(y) + g(y)\right) I_0\left(\frac{\beta}{2}\sqrt{\phi_\alpha(t)^2 - \frac{(y-x)^2}{c^2}}\right) dy
\end{aligned}
\end{equation}
where $I_{0}$ is the modified Bessel function of the first kind (order 0).

 This shows that conformable evolution equations in general \(D^\alpha_\psi u = Au\) are equivalent to classical problems \( \partial_\tau u = Au \) under \(\tau = \phi_\alpha(t)\), with generators scaled by \(\psi_\alpha(t)\). The proof relies on the chain rule and the continuity of the semigroup, revealing that the conformable framework preserves the semigroup structure but introduces time-dependent scaling. Unlike fractional calculus, no memory effects emerge, only a nonlinear time reparametrization.

\subsection{Dynamical Systems}

The fundamental equivalence established by Theorem \ref{1} naturally extends to dynamical systems described by conformable derivatives. Consider a general dynamical system of the form:
\begin{equation}\label{dynsys}
D_\psi^\alpha \mathbf{x}(t) = \mathbf{f}(\mathbf{x}(t)), \quad \mathbf{x}(0) = \mathbf{x}_0,
\end{equation}
where $\mathbf{x}(t) \in \mathbb{R}^n$, $\mathbf{f}: \mathbb{R}^n \to \mathbb{R}^n$ is a vector field, and $D_\psi^\alpha$ denotes the conformable derivative of order $\alpha$ with respect to the weight function $\psi_\alpha(t)$.

\begin{proposition}
The conformable dynamical system \eqref{dynsys} is equivalent to the classical dynamical system:
\[
\frac{d\mathbf{y}}{d\tau} = \mathbf{f}(\mathbf{y}(\tau)), \quad \mathbf{y}(0) = \mathbf{x}_0,
\]
by the change of variable $\tau = \phi_\alpha(t) = \int_0^t \frac{1}{\psi_\alpha(s)} ds$ and $\mathbf{y}(\tau) = \mathbf{x}(t)$.
\end{proposition}

\begin{proof}
From Theorem \ref{1}, we have:
\[
D_\psi^\alpha \mathbf{x}(t) = \psi_\alpha(t) \frac{d\mathbf{x}}{dt} = \psi_\alpha(t) \left( \frac{d\mathbf{y}}{d\tau} \cdot \frac{d\tau}{dt} \right) = \frac{d\mathbf{y}}{d\tau}.
\]
Substituting into \eqref{dynsys}, we obtain:
\[
\frac{d\mathbf{y}}{d\tau} = \mathbf{f}(\mathbf{y}(\tau)),
\]
which is a classical dynamical system in the variable $\tau$. The initial condition transforms trivially since $\tau(0) = 0$ and $\mathbf{y}(0) = \mathbf{x}(0) = \mathbf{x}_0$.
\end{proof}

\begin{example}[Conformable Lorenz System]
Let's revisit the example of the Lorenz system with conformable derivative:
\begin{equation}\label{lorenzconf}
\begin{cases}
D_\psi^\alpha x = \sigma(y - x) \\
D_\psi^\alpha y = \rho x - y - xz \\
D_\psi^\alpha z = xy - \beta z
\end{cases}
\end{equation}
with $\psi_\alpha(t) = t^{1-\alpha}$.

Apply the fundamental relation \eqref{psi_C}: $D_\psi^\alpha x = \psi_\alpha(t) \frac{dx}{dt} = t^{1-\alpha} \frac{dx}{dt}$.
The system therefore becomes:
\[
\begin{cases}
t^{1-\alpha} \frac{dx}{dt} = \sigma(y - x) \\
t^{1-\alpha} \frac{dy}{dt} = \rho x - y - xz \\
t^{1-\alpha} \frac{dz}{dt} = xy - \beta z
\end{cases}
\]

Divide each equation by $t^{1-\alpha}$ (for $t > 0$):
\[
\begin{cases}
\frac{dx}{dt} = \frac{\sigma(y - x)}{t^{1-\alpha}} \\
\frac{dy}{dt} = \frac{\rho x - y - xz}{t^{1-\alpha}} \\
\frac{dz}{dt} = \frac{xy - \beta z}{t^{1-\alpha}}
\end{cases}
\]

Now introduce the change of variable $\tau = \phi_\alpha(t) = \int_0^t \frac{1}{\psi_\alpha(s)} ds = \int_0^t s^{\alpha-1} ds = \frac{t^\alpha}{\alpha}$.
By the chain rule, we have:
\[
\frac{d}{dt} = \frac{d\tau}{dt} \cdot \frac{d}{d\tau} = \frac{1}{\psi_\alpha(t)} \frac{d}{d\tau} = t^{\alpha-1} \frac{d}{d\tau}
\]

Substitute this relation into the system:
\[
\begin{cases}
t^{\alpha-1} \frac{dx}{d\tau} = \frac{\sigma(y - x)}{t^{1-\alpha}} \\
t^{\alpha-1} \frac{dy}{d\tau} = \frac{\rho x - y - xz}{t^{1-\alpha}} \\
t^{\alpha-1} \frac{dz}{d\tau} = \frac{xy - \beta z}{t^{1-\alpha}}
\end{cases}
\]

Multiply each equation by $t^{1-\alpha}$:
\[
\begin{cases}
\frac{dx}{d\tau} = \sigma(y - x) \\
\frac{dy}{d\tau} = \rho x - y - xz \\
\frac{dz}{d\tau} = xy - \beta z
\end{cases}
\]

We thus obtain the classical Lorenz system:
\begin{equation}\label{lorenzclass}
\begin{cases}
\frac{dx}{d\tau} = \sigma(y - x) \\
\frac{dy}{d\tau} = \rho x - y - xz \\
\frac{dz}{d\tau} = xy - \beta z
\end{cases}
\end{equation}

This result demonstrates that the conformable Lorenz system \eqref{lorenzconf} is mathematically equivalent to the classical system \eqref{lorenzclass} by the simple time scaling change $\tau = \frac{t^\alpha}{\alpha}$.

The solution of the conformable system is expressed in terms of the classical solution $\mathbf{y}(\tau) = (x(\tau), y(\tau), z(\tau))^T$:
\[
\mathbf{x}(t) = \mathbf{y}\left( \frac{t^\alpha}{\alpha} \right)
\]

This transformation preserves all qualitative properties of the system (fixed points, stability, chaos) but modifies the time scale of the evolution. For example, a period $T$ in conformable time $\tau$ corresponds to a time $t = (\alpha T)^{1/\alpha}$ in the original time.
\end{example}

\begin{remark}
This equivalence has important implications for qualitative analysis:
\begin{itemize}
\item Fixed points and their stability are identical in both formulations
\item Lyapunov exponents are preserved (up to a temporal scaling factor)
\item Chaotic behavior, understood in the sense of sensitivity to initial conditions, is preserved under smooth time reparametrization
\item Bifurcations occur at the same critical parameter values
\end{itemize}
The conformable derivative therefore does not modify the qualitative nature of the dynamics, but only its time scale.
\end{remark}

\begin{proposition}[Invariance of qualitative properties]
Let $\mathbf{x}(t)$ be a solution of the conformable system \eqref{dynsys} and $\mathbf{y}(\tau)$ the solution of the associated classical system. Then:
\begin{enumerate}
\item $\mathbf{x}^*$ is a fixed point of \eqref{dynsys} if and only if $\mathbf{y}^* = \mathbf{x}^*$ is a fixed point of the classical system
\item Linear stability is identical: the eigenvalues of the Jacobian matrix are the same
\item Periodic orbits are preserved (with rescaled period)
\item The chaos property is invariant
\end{enumerate}
\end{proposition}

\begin{proof}
The proof follows directly from the established differential equivalence. The change of variable $\tau = \phi_\alpha(t)$ being an increasing diffeomorphism, it preserves temporal order and all qualitative properties of the solutions.
\end{proof}

This analysis demonstrates that the use of conformable derivatives in dynamical systems does not alter the underlying dynamical structure of the system, but simply constitutes a mathematical reformulation equivalent to a nonlinear temporal reparametrization.

\section{Effectiveness in Modeling Real-World Problems}

\subsection{The misconception of Physical Relevance of Conformable Derivatives}

A key component of the conformable derivative is the time transformation by the function:
\[
\tau(t) = \frac{t^{\alpha}}{\alpha}, \quad \text{for } 0 < \alpha < 1.
\]

This function is strictly increasing and concave on \( \mathbb{R}_+ \). More precisely, for all \( t > 1 \), we have:
\[
\tau(t) \leq t.
\]

This inequality shows that the reparametrized (or “conformable”) time $ \tau(t) $ runs slower than real time $t$ for $t>1$. Consequently, any dynamics described in terms of conformable time appears artificially slowed down when observed in real time.

\begin{proposition}[Linearity Effect from Time Reparametrization] The concavity of the time transformation \( t \mapsto \frac{t^{\alpha}}{\alpha} \) in conformable derivatives introduces an artificial acceleration of the system's evolution. This creates the illusion of linear behavior in otherwise nonlinear phenomena, not by modeling physical memory effects, but simply by distorting the time scale.
\end{proposition}

\noindent
This fact is crucial when applying conformable derivatives to experimental models. What appears to be a better fit to data may simply reflect a time reparametrization, and not an accurate representation of the involved physical processes.

\vspace{1em}
\noindent
\textbf{Remark.} Unlike Caputo or Riemann-Liouville derivatives, which integrate memory and hereditary effects over the past, The conformable derivative does not incorporate explicit nonlocal memory terms. Therefore, it cannot capture the non-local behavior essential to many complex systems, such as viscoelastic materials or anomalous diffusion.

\subsection{The Lorenz Attractor and the Conformable Derivative}

We consider the classical Lorenz system defined by:
\[
\begin{cases}
\frac{dx}{dt} = \sigma(y - x) \\
\frac{dy}{dt} = \rho x - y - xz \\
\frac{dz}{dt} = xy - \beta z
\end{cases}
\]
with parameters \( \sigma = 10 \), \( \rho = 28 \), \( \beta = \frac{8}{3} \), and initial conditions \( x(0) = y(0) = z(0) = 1 \). This system generates the famous chaotic Lorenz attractor.

To approximate the conformable derivative of order \( \alpha \in (0,1) \), we use:
\[
D^{\alpha}f(t) \approx t^{1 - \alpha} \cdot \frac{df}{dt}.
\]
The conformable Lorenz system becomes:
\[
\begin{cases}
\frac{dx}{dt} = \frac{1}{t^{1 - \alpha}} \sigma(y - x) \\
\frac{dy}{dt} = \frac{1}{t^{1 - \alpha}} (\rho x - y - xz) \\
\frac{dz}{dt} = \frac{1}{t^{1 - \alpha}} (xy - \beta z)
\end{cases}
\]

The following figures (Figs. \ref{fig1}-\ref{fig4}) illustrate the evolution of the Lorenz attractor using both the standard derivative and the conformable derivative with \( \alpha = 0.9 \).

\begin{figure}[H]
    \centering
        \includegraphics[width=\textwidth]{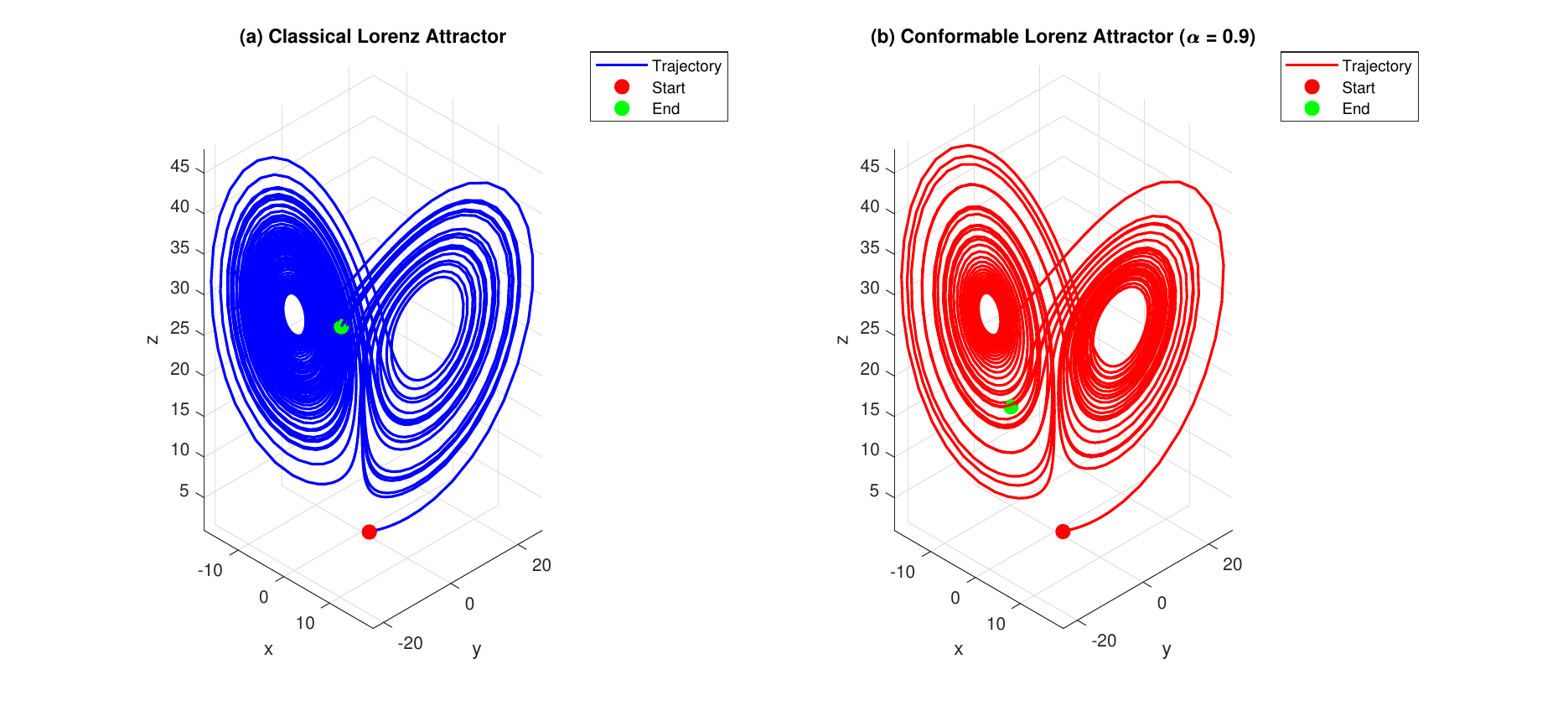}
        \caption{Phase space trajectories of the Lorenz system. (a) The classical Lorenz attractor. (b) The attractor under the conformable derivative formulation with $  \alpha = 0.9 $, exhibiting a temporally distorted but topologically equivalent structure.}
        \label{fig1}
\end{figure}

\begin{figure}[H]
        \centering
        \includegraphics[width=\textwidth]{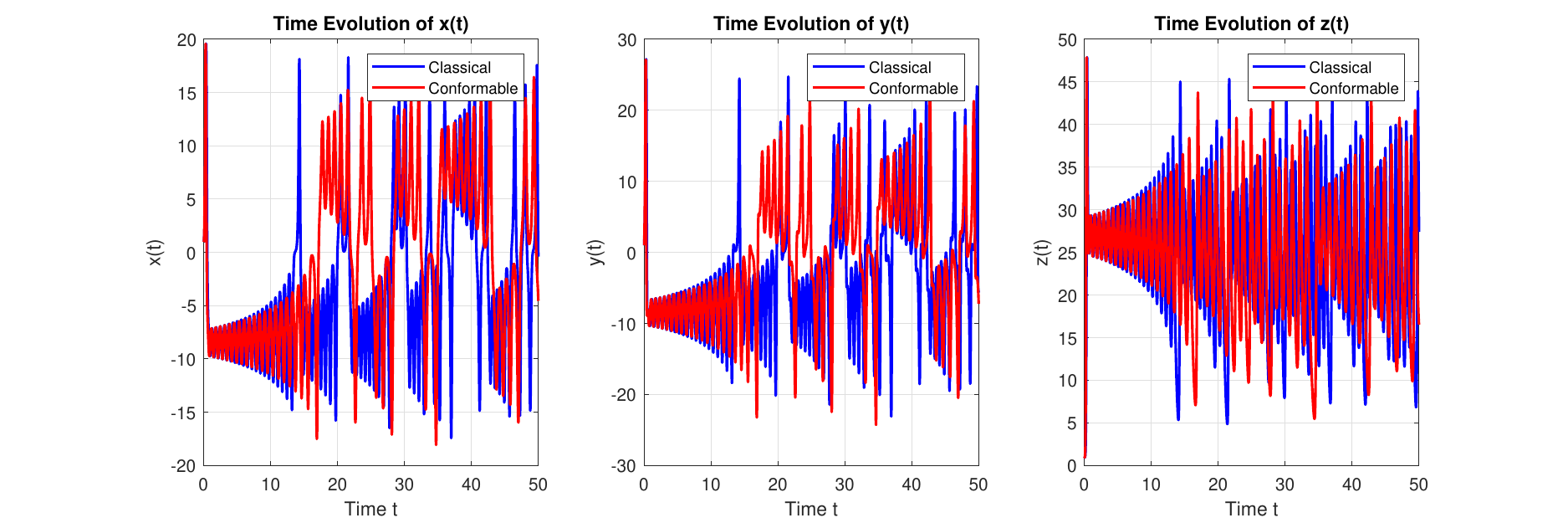}
        \caption{Time series of the Lorenz system variables $x(t)$, $y(t)$, and $z(t)$. The conformable system (red) replicates the dynamics of the classical system (blue) on a rescaled time axis, confirming the equivalence under time reparametrization.}
                \label{fig2}
\end{figure}

\begin{figure}[H]
        \centering
        \includegraphics[width=\textwidth]{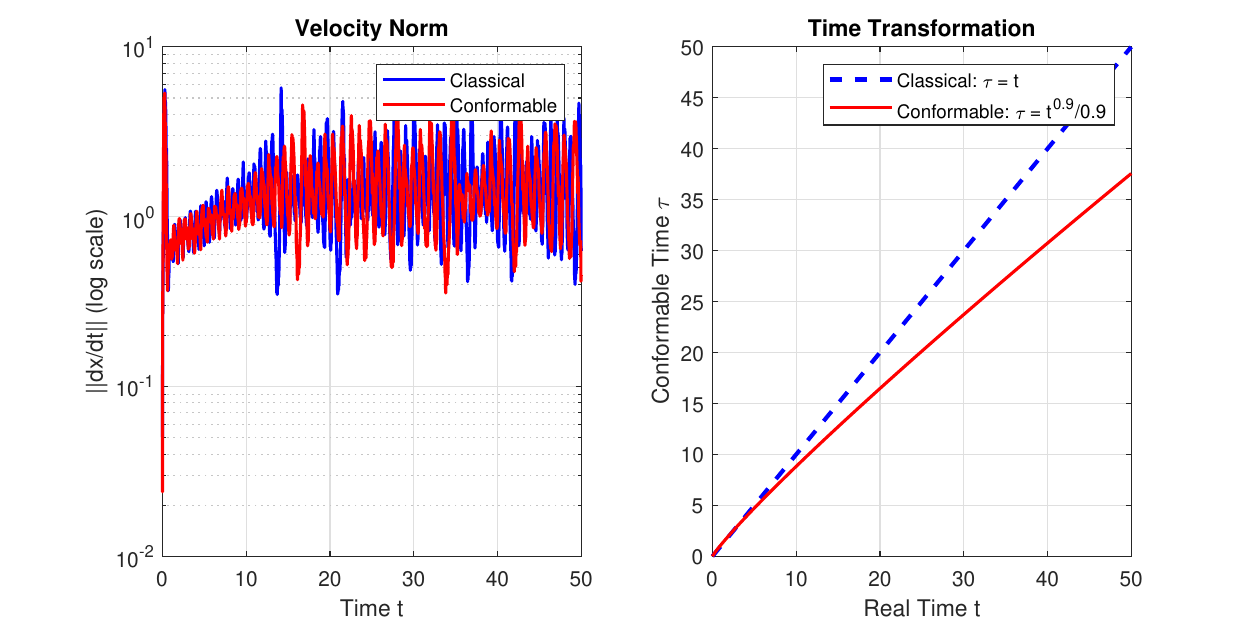}
        \caption{Dynamic characteristics of the conformable Lorenz system. (a) The norm of the system velocity $\|x^{\prime}(t)\|$, showing initial acceleration. (b) The nonlinear time transformation $\tau = t^\alpha/\alpha$ for $  \alpha = 0.9$, which is the source of the observed kinematic distortion.}
          \label{fig3}
\end{figure}

\begin{figure}[H]
        \centering
        \includegraphics[width=\textwidth]{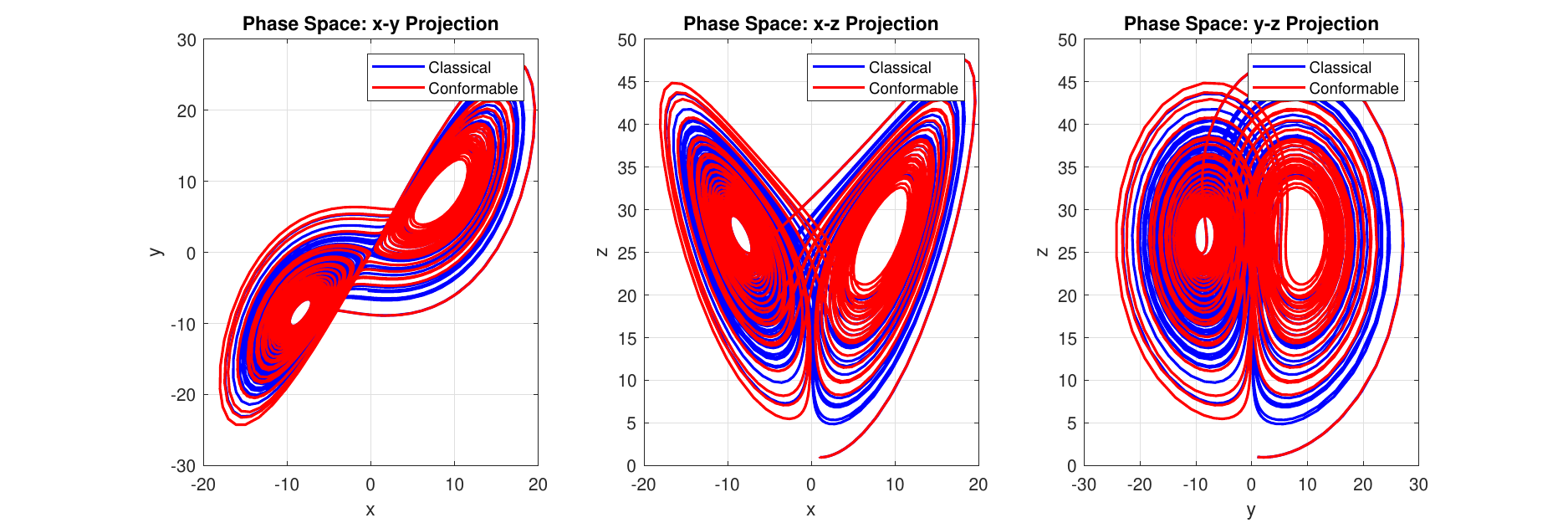}
        \caption{Two-dimensional phase space projections $(x,y)$, $(x,z)$, and $(y,z)$ of the Lorenz attractor. The geometric structure is preserved under the conformable derivative, demonstrating that the transformation affects only the temporal, not the spatial, evolution.}
                \label{fig4}
\end{figure}

\section*{Interpretation of Numerical Results}

The numerical simulations presented in Figures \ref{fig1} to \ref{fig4} strikingly corroborate our theoretical analysis and provide a visual illustration of the fundamental equivalence between the conformable formalism and classical dynamics under temporal re-parametrization.

Figure \ref{fig1}-(a) shows the classical Lorenz attractor, with its characteristic "butterfly" structure and its sensitivity to initial conditions. Figure \ref{fig1}-(b), representing the conformable system with $\alpha = 0.9$, appears at first glance different: the trajectory seems more "stretched" and evolves more rapidly in phase space. This superficial difference could be mistakenly interpreted as the emergence of new dynamical behavior.

However, this interpretation would be incorrect. Figures \ref{fig2}(a-c) reveal the true nature of this difference: the temporal evolution of the variables $(x, y, z)$ shows that the conformable system (red line) follows exactly the same dynamics as the classical system (blue line), but on an accelerated time scale. This acceleration is particularly marked at short times, as confirmed by Figure \ref{fig3}-(a) which shows a higher velocity norm for the conformable system in the first instants.

The key to interpretation lies in Figure \ref{fig3}-(b), which illustrates the time transformation $\tau = t^\alpha/\alpha$. The nonlinearity of this transformation, particularly at low values of $t$, explains all the observed differences between the two representations. The conformable system is not fundamentally different; it is simply the classical system "seen" through a nonlinear clock.

Finally, Figure \ref{fig4}, presenting the two-dimensional projections of the attractor, confirms that the fundamental geometric structure of the attractor is preserved. The trajectories in the $(x,y)$, $(x,z)$ and $(y,z)$ planes follow the same paths in phase space, confirming that the temporal reparametrization does not affect the qualitative structure of the dynamics.

These visual results convincingly demonstrate that the conformable derivative does not create new physics but simply provides an alternative representation of classical dynamics, equivalent to a nonlinear change of time scale. This interpretation is consistent with our theoretical analysis and challenges the claim that the conformable formalism introduces truly fractional or non-local behaviors.

\begin{proposition}[Temporal Rescaling of Chaotic Dynamics] The conformable Lorenz system does not exhibit a fundamentally new type of chaos. Instead, the observed difference stems from the concave time scaling, which distorts the underlying dynamics. The effect is not physical memory but an illusion of nonlinearity.
\end{proposition}

Thus, any apparent fractional behavior in conformable chaos is a consequence of time distortion rather than fractional calculus in the integral sense (e.g., Caputo or Riemann-Liouville derivatives).

\subsection{The Fractional Lorenz Attractor and the Caputo Derivative}

To complete our comparative study, we now consider the Lorenz system with a true fractional derivative in the Caputo sense. Unlike the conformable derivative, the Caputo derivative introduces long-term memory into the system, fundamentally altering its dynamics.

The Caputo derivative of order $\alpha$ is defined by:
\[
D^\alpha_t f(t) = \frac{1}{\Gamma(1-\alpha)} \int_0^t \frac{f'(\tau)}{(t-\tau)^\alpha} d\tau, \quad 0 < \alpha < 1.
\]

The fractional Lorenz system is written:
\begin{equation}\label{lorenz_caputo}
\begin{cases}
D^\alpha_t x = \sigma(y - x) \\
D^\alpha_t y = \rho x - y - xz \\
D^\alpha_t z = xy - \beta z
\end{cases}
\end{equation}
with the same parameters and initial conditions as before.

\subsubsection*{Numerical Method for the Caputo Derivative}

Numerical solution of the fractional system requires specialized methods. We use the Adams-Bashforth-Moulton Predictor-Corrector algorithm for fractional differential equations:

\begin{equation*}
x_{n+1} = x_0 + \frac{1}{\Gamma(\alpha)} \sum_{j=0}^n \beta_{j,n+1} \cdot \sigma(y_j - x_j)
\end{equation*}
\begin{equation*}
y_{n+1} = y_0 + \frac{1}{\Gamma(\alpha)} \sum_{j=0}^n \beta_{j,n+1} \cdot (\rho x_j - y_j - x_j z_j)
\end{equation*}
\begin{equation*}
z_{n+1} = z_0 + \frac{1}{\Gamma(\alpha)} \sum_{j=0}^n \beta_{j,n+1} \cdot (x_j y_j - \beta z_j)
\end{equation*}

where the coefficients $\beta_{j,n+1}$ capture the memory effect of the fractional derivative (see \ref{fig5} and \ref{fig6}).

\begin{figure}[H]
        \centering
        \includegraphics[width=\textwidth]{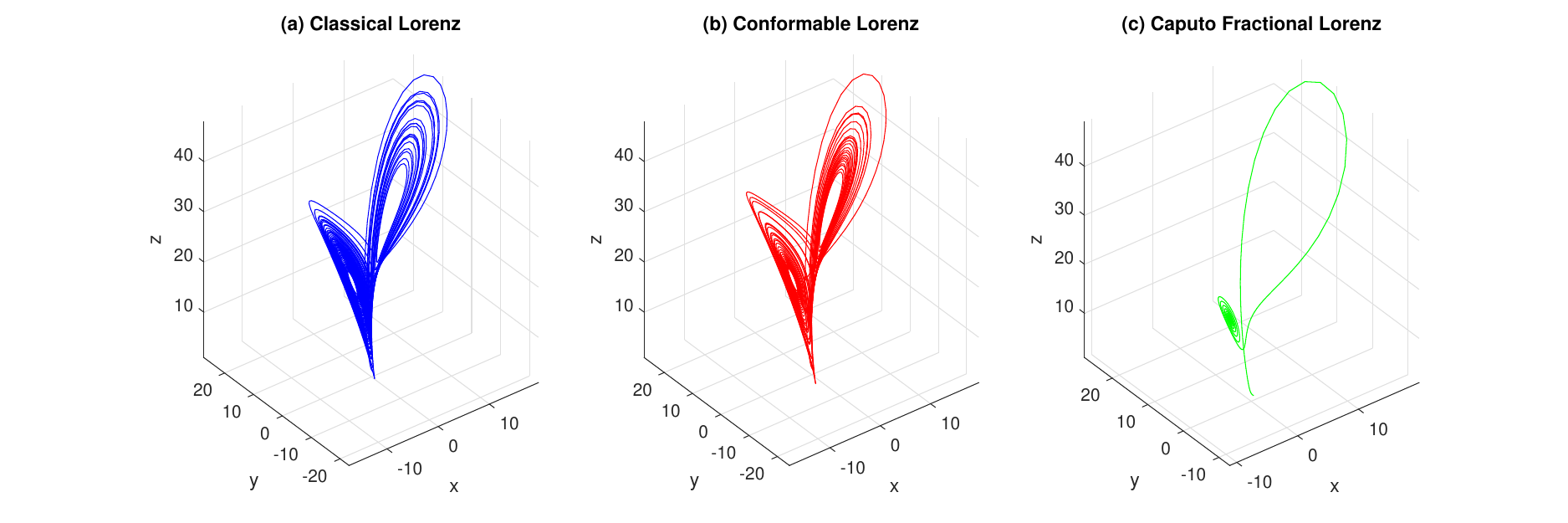}
        \caption{Comparative analysis of Lorenz attractors for $\alpha = 0.9$. (a) Classical derivative (reference). (b) Conformable derivative: a kinematically distorted version of (a). (c) Caputo fractional derivative: a structurally altered attractor exhibiting genuine memory effects and fractal roughness.}
        \label{fig5}
\end{figure}

\begin{figure}[H]
        \centering
        \includegraphics[width=0.9\textwidth]{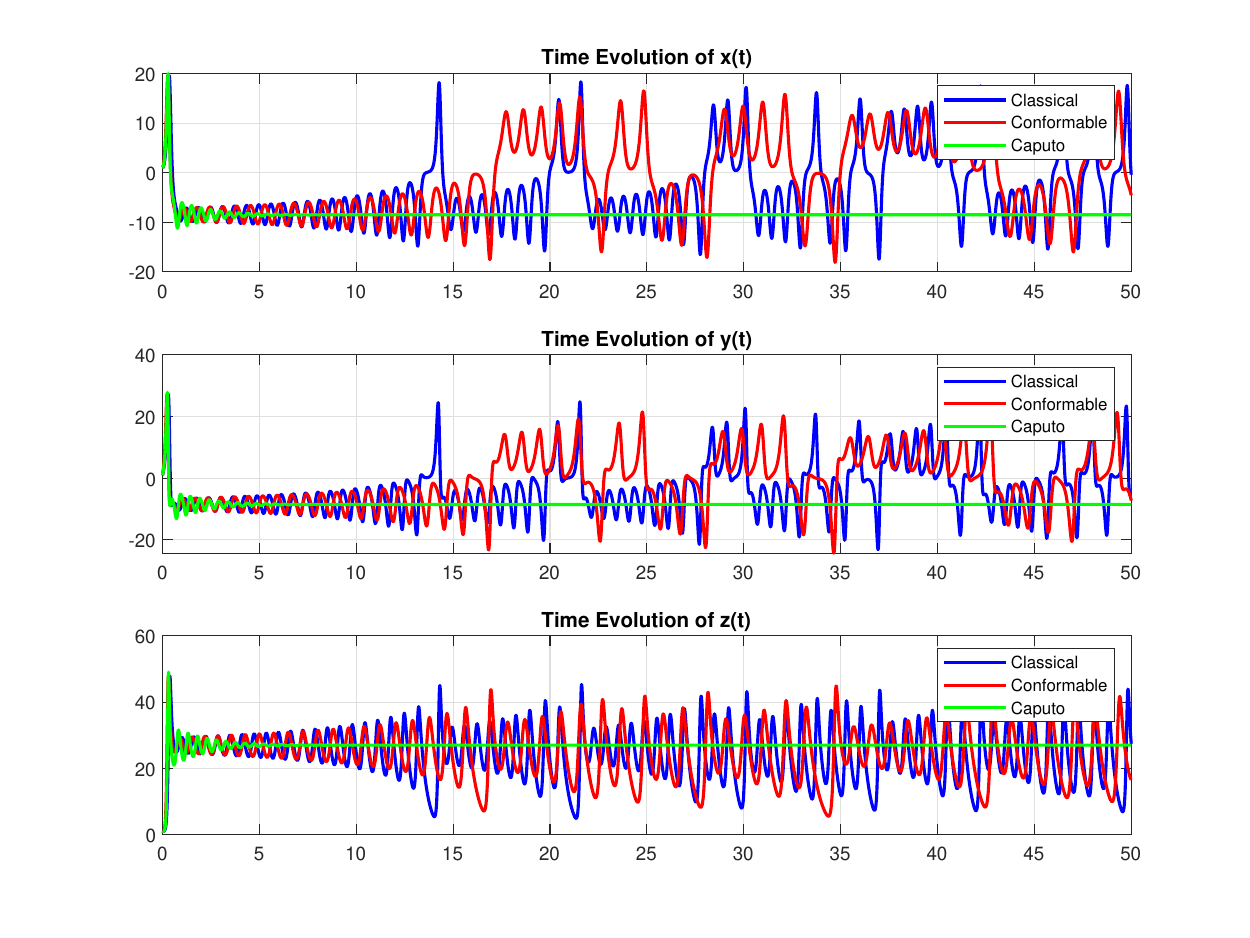}
        \caption{Temporal evolution of the variable $x(t)$ for the classical (blue), conformable (red), and Caputo fractional (green) Lorenz systems. The conformable trajectory is a time-warped version of the classical one, while the Caputo system shows authentic memory effects, including state persistence and amplitude modulation.}
        \label{fig6}
\end{figure}

\section*{Interpretation of Numerical Results}

The numerical simulations of the classical, conformable, and fractional (Caputo) Lorenz systems reveal fundamental differences that confirm our theoretical analysis and illuminate the distinct nature of these three formulations.

\begin{description}
\item Figure \ref{fig5} presents an immediately revealing visual comparison. The classical attractor (\ref{fig5}-a) exhibits the well-known "butterfly" structure, with its two characteristic lobes and the chaotic trajectory alternating between them. The conformable attractor (\ref{fig5}-b) shows a geometrically similar structure but \textbf{temporally distorted}: the loops are tighter and the evolution appears accelerated, particularly in the early phases. This distortion is purely kinematic and results from the reparametrization $\tau = t^\alpha/\alpha$.

In contrast, the Caputo attractor (\ref{fig5}-c) reveals a \textbf{profound structural modification}. The geometry of the attractor is altered: the trajectories show a fractal "roughness" characteristic of systems with memory, the lobes are less distinct, and the overall spatial organization differs. This transformation is not a simple temporal distortion but an authentic \textbf{modification of the underlying dynamics}.

\item Figure \ref{fig6}, showing the temporal evolution of the variable $x(t)$, is particularly illuminating. The classical and conformable systems essentially follow the same evolution (blue and red lines), shifted in time according to the transformation $\tau(t)$. Their oscillatory behavior is synchronized, confirming that the conformable derivative preserves the dynamics while changing its time scale.

The Caputo system (green line), on the other hand, exhibits radically different behavior:
\begin{itemize}
\item[•] \textbf{Memory effect}: The trajectory shows increased persistence in each state, with slower transitions between the lobes of the attractor
\item[•] \textbf{Amplitude modulation}: The oscillations show a modulated envelope characteristic of systems with long-term memory
\item[•] \textbf{Multiple scales}: Natural emergence of different time scales due to the integro-differential nature of the operator
\end{itemize}

\end{description}

This analysis clearly demonstrates that:

\begin{enumerate}
\item The conformable derivative only produces a \textbf{temporal reparametrization} of the classical system

\item The Caputo derivative introduces a \textbf{genuine physical modification} of the dynamics

\item The purported "fractional effects" of conformable derivatives are \textbf{modeling artifacts}

\item Only integro-differential derivatives (Caputo, RL) capture authentic memory phenomena
\end{enumerate}

This distinction is crucial for faithful physical modeling. The conformable derivative may be useful as a computational tool, but should not be confused with a true fractional derivative for modeling systems with memory.

\section*{Synthetic Conclusion}

This work clarifies the mathematical nature of the conformable derivative by showing that it can be systematically interpreted as a classical derivative under a nonlinear time reparametrization. Within this framework, conformable models preserve the qualitative structure of the underlying classical dynamics while modifying the temporal scale of evolution.

In contrast with established fractional derivatives such as those of Caputo or Riemann-Liouville, the conformable derivative does not incorporate nonlocal memory effects. Consequently, its use is not suited for modeling phenomena where hereditary behavior plays a fundamental role.

These observations suggest that conformable operators are best understood as weighted or reparametrized derivatives rather than genuine fractional derivatives. When employed with this interpretation, they may serve as useful computational tools. However, for the modeling of memory-dependent processes, classical fractional calculus remains the appropriate framework.

\end{document}